\documentclass{amsart}

\usepackage{amsmath}
\usepackage{amssymb}
\usepackage{amsthm}
\usepackage{tikz-cd}
\usepackage{mathrsfs}
\usepackage{hyperref}
\usepackage{enumerate}
\usepackage{todonotes}

\theoremstyle{plain}
\newtheorem{theorem}{Theorem}[section]
\newtheorem{lemma}[theorem]{Lemma}

\newtheorem{proposition}[theorem]{Proposition}

\newtheorem{conjecture}[theorem]{Conjecture}
\theoremstyle{definition}

\allowdisplaybreaks[3]

\title[\centering\parbox{0.8\linewidth}{\centering Nonexistence of perfect $2$-error-correcting Lee codes in certain dimensions}]{Nonexistence of perfect $2$-error-correcting Lee codes in certain dimensions}
\author{Dongryul Kim}
\email{\href{mailto:dkim04@college.harvard.edu}{{\tt dkim04@college.harvard.edu}}}
\address{Harvard College, Cambridge, MA, 02138}

\begin{document}

\maketitle

\begin{abstract}
  The Golomb--Welch conjecture states that there are no perfect $e$-error-correcting codes in $\mathbb{Z}^n$ for $n \ge 3$ and $e \ge 2$. In this note, we prove the nonexistence of perfect $2$-error-correcting codes for a certain class of $n$, which is expected to be infinite. This result further substantiates the Golomb--Welch conjecture.
\end{abstract}

\section{Introduction}

For an integer $q \ge 2$, consider the space $(\mathbb{Z}/q\mathbb{Z})^n$ equipped with the Lee metric $d$ given by
\[
  d(\mathbf{x}, \mathbf{y}) = \sum_{i=1}^{n} \min\{ \lvert x_i - y_i \rvert, q - \lvert x_i - y_i \rvert \}.
\]
An \emph{$e$-error-correcting Lee code} is a subset $C \subseteq (\mathbb{Z} / q\mathbb{Z})^n$ such that any two distinct elements of $C$ have distance at least $2e+1$. An $e$-error-correcting Lee code $C$ is further called a \emph{perfect $e$-error-correcting Lee code} if for each $x \in (\mathbb{Z} / q\mathbb{Z})^n$, there exists a unique element $c \in C$ such that $d(x, c) \le e$. A perfect $e$-error-correcting Lee code in $(\mathbb{Z}/q\mathbb{Z})^n$ is also called simply a $PL(n,e,q)$-code. 

There is an equivalent description of error-correcting Lee codes that uses the language of tilings. Consider the \emph{Lee sphere} 
\[
  S(n,e,q) = \{ \mathbf{x} \in (\mathbb{Z}/q\mathbb{Z})^n : d(\mathbf{x}, \mathbf{0}) \le e \}
\]
of radius $e$. An $e$-error-correcting Lee code is a subset $C \subseteq (\mathbb{Z}/q\mathbb{Z})^n$ such that for any $\mathbf{x} \neq \mathbf{y}$ in $C$, the two spheres $\mathbf{x} + S(n,e,q)$ and $\mathbf{y} + S(n,e,q)$ are disjoint. Thus it can be naturally identified with a translational packing of $S(n,e,q)$ in $(\mathbb{Z}/q\mathbb{Z})^n$. A perfect $e$-error-correcting Lee code then corresponds to a translational tiling of $(\mathbb{Z}/q\mathbb{Z})^n$ by $S(n,e,q)$. 

If $q \ge 2e + 1$, then the natural projection map $\mathbb{Z}^n \to (\mathbb{Z}/q\mathbb{Z})^n$ restricts to a bijection from 
\[
  S(n,e) = \{ \mathbf{x} \in \mathbb{Z}^n : \lvert x_1 \rvert + \lvert x_2 \rvert + \dots + \lvert x_n \rvert \le e \}
\]
to $S(n,e,q)$. Any tiling of $(\mathbb{Z}/q\mathbb{Z})^n$ by $S(n,e,q)$ will then pull back via the projection to a tiling of $\mathbb{Z}^n$ by $S(n,e)$. Let us call a subset $C \subseteq \mathbb{Z}^n$ a \emph{perfect $e$-error-correcting Lee code} in $\mathbb{Z}^n$, or simply a $PL(n,e)$-code, if the translates of $S(n,e)$ centered at vectors of $C$ form a tiling of $\mathbb{Z}^n$. Then a $PL(n,e,q)$-code induces a $PL(n,e)$-code that is a disjoint union of cosets of $q\mathbb{Z}^n \subset \mathbb{Z}^n$. Conversely, any such $PL(n,e)$-code clearly comes from a $PL(n,e,q)$-code. We restate this in the following proposition.

\begin{proposition}
  For $q \ge 2e+1$, there exists a natural bijection between $PL(n,e,q)$-codes and $PL(n,e)$-codes that is a union of cosets of $q\mathbb{Z}^n \subset \mathbb{Z}^n$, given by taking the image or the inverse image with respect to the projection map $\mathbb{Z}^n \to (\mathbb{Z}/q\mathbb{Z})^n$. 
\end{proposition}

\noindent Thus to know all about $PL(n,e,q)$-codes, it suffices to study $PL(n,e)$-codes. 

Error-correcting codes in the Lee metric have been first investigated by Golomb and Welch \cite{GW70}. In the paper, they explicitly construct $PL(1,e,2e+1)$-codes, $PL(2,e,2e^2 + 2e + 1)$-codes, and $PL(n,1,2n+1)$-codes. On the other hand, they conjecture the nonexistence of perfect Lee codes for other $n$ and $e$. 

\begin{conjecture}
  For $n \ge 3$ and $e \ge 2$, there exist no $PL(n,e)$-codes.
\end{conjecture}

The case when $e$ is ``large'' compared to $n$ is studied extensively in the literature. Golomb and Welch \cite{GW70} proved using a compactness argument that for each $n \ge 3$, there exists a sufficiently large $\rho_n$ such that there exist no $PL(n,e)$-codes for each $e \ge \rho_n$. An effective form of this theorem, that $PL(n,e,q)$-codes do not exist for $3 \le n \le 5, e \ge n-1, q \ge 2e+1$ and $n \ge 6, e \ge \frac{\sqrt{2}}{2} n - \frac{3}{4}\sqrt{2} - \frac{1}{2}, q \ge 2e+1$, was subsequently shown by Post \cite{Pos75}. Lepist\"o \cite{Lep81} improved the bound asymptotically and obtained the following theorem. 

\begin{theorem}
  For any $n, e, q$ satisfying $n < (e+2)^2 / 2.1$ and $e \ge 285$ and $q \ge 2e + 1$, there exist no $PL(n,e,q)$-codes.
\end{theorem}


Another direction of approach is to focus on small $n$. Gravier, Mollard, and Payan \cite{GMP98} showed the nonexistence of $PL(3,e)$-codes by analyzing possible local configurations. Later a computer-based proof of the nonexistence of $PL(4,e)$-codes was given by \v{S}pacapan \cite{Spa07}, and Horak \cite{Hor09_2} further extended the theorem to prove nonexistence of $PL(n,e)$-codes for $3 \le n \le 5$ and $e \ge 2$. In recent years, the case $e = 2$ has been investigated for reasonably small $n$. For $n = 5, 6$, Horak \cite{Hor09} showed that $PL(5,2)$-codes and $PL(6,2)$-codes do not exist, and Horak and Gros\v{e}k \cite{HG14} further showed using a computer that for $7 \le n \le 12$ there are no linear $PL(n,2)$-codes, i.e., $PL(n,2)$-codes that is a lattice in $\mathbb{Z}^n$. 

In this note, we continue along this line and provide a number theoretic condition under which $PL(n,2)$-codes do not exist. In particular, we prove the following theorem. 

\begin{theorem} \label{thm:main}
  Suppose $p = 2n^2 + 2n + 1$ is prime. Let $a$ be the smallest positive integer for which $p \mid 4^a + 4n + 2$ and $b$ be the smallest positive integer for which $p \mid 4^b - 1$. (For convenience let $a = \infty$ if there is no $a$ with $p \mid 4^a + 4n + 2$.) If the equation $a(x+1) + by = n$ has no nonnegative integer solutions, then $PL(n,2)$-codes do not exist. For instance, there are no $PL(n,2)$-codes for $n = 5, 7, 9, 12, 14, 17, \ldots$. 
\end{theorem}

To illustrate the strength of this theorem, we provide numerical data concerning the number of $n$ to which the theorem can be applied. As in Table~\ref{tab:num_data}, if $2n^2 + 2n + 1$ is indeed prime, in most cases the second condition about the equation having no nonnegative solutions is also satisfied. It is reasonable to expect that there are infinitely many $n$ such that $2n^2 + 2n + 1$ is prime, although it is far from being proved. This is a special case of the Bunyakovsky conjecture, and moreover the heuristics of the Bateman--Horn conjecture \cite{BH62} expects there to be asymptotically $C x / \log x$ such $n \le x$ for some absolute constant $C$. 

\begin{table}
  \centering
    \begin{tabular}{c|cc}
      \hspace{1em} $x$ \hspace{1em} & \begin{tabular}{c} \# of $n \le x$ with \\ $2n^2 + 2n + 1$ prime \end{tabular} & \begin{tabular}{c} \# of $n \le x$ to which \\ Theorem~\ref{thm:main} can be applied \end{tabular} \\
      \hline
      $10^1$ & $6$ & $4$ \\
      $10^2$ & $36$ & $34$ \\
      $10^3$ & $225$ & $222$ \\
      $10^4$ & $1645$ & $1642$ \\
      $10^5$ & $12706$ & $12702$
    \end{tabular}
    \caption{The number of $n$ to which Theorem~\ref{thm:main} can be applied}
    \label{tab:num_data}
\end{table}

The condition $2n^2 + 2n + 1 = \lvert S(n,2) \rvert$ being prime is included in order to use a result that allows us to translate the tiling problem to a purely algebraic problem. The following theorem is proved in \cite{Sze98}. 

\begin{theorem} \label{thm:til_alg}
  Let $T \subset \mathbb{Z}^n$ be a finite subset of prime size $p$, and suppose that $T - T \subset \mathbb{Z}^n$ generates $\mathbb{Z}^n$ as an abelian group. Then there exists a tiling of $\mathbb{Z}^n$ by translates of $T$ if and only if there exists a homomorphism $\phi : \mathbb{Z}^n \to \mathbb{Z}/p\mathbb{Z}$ that restricts to a bijection from $T$ to $\mathbb{Z}/p\mathbb{Z}$. 
\end{theorem}

This research did not receive any specific grant from funding agencies in the public, commercial, or not-for-profit sectors.

\section{Proof of Theorem~\ref{thm:main}}

In this section, we let $2n^2 + 2n + 1 = p$ be a prime. Since $\mathbb{Z}^n$ is a free abelian group generated by the unit vectors $e_1, \dots, e_n$, a homomorphism $\phi : \mathbb{Z}^n \to \mathbb{Z}/p\mathbb{Z}$ is determined uniquely by the values $x_i = \phi(e_i)$ for $1 \le i \le n$. Then $\phi$ restricting to a bijection from $S(n,2)$ to $\mathbb{Z}/p\mathbb{Z}$ is equivalent to the sets
\[
  \{0\}, \quad \{\pm x_i\}_{1 \le i \le n}, \quad \{\pm 2 x_i\}_{1 \le i \le n}, \quad \{\pm x_i \pm x_j \}_{1 \le i < j \le n}
\]
forming a partition of $\mathbb{Z}/p\mathbb{Z}$. 

Suppose that such $x_1, \dots, x_n \in \mathbb{Z}/p\mathbb{Z}$ exist. The sum of $2k$-th powers of all the elements is 
\begin{align*}
  \sum_{i=1}^{n} \bigl( & x_i^{2k} + (-x_i)^{2k} + (2x_i)^{2k} + (-2x_i)^{2k} \bigr) \\
  & \qquad + \sum_{1 \le i < j \le n}^{} \bigl( (x_i + x_j)^{2k} + (x_i - x_j)^{2k} + (-x_i + x_j)^{2k} + (-x_i - x_j)^{2k} \bigr) \\
  &= 2(4^k + 1) \sum_{i=1}^{n} x_i^{2k} + \sum_{1 \le i < j \le n}^{} 4 \sum_{t=0}^{k} \binom{2k}{2t} x_i^{2t} x_j^{2(k-t)} \\
  &= (2^{2k+1} + 4(n-1) + 2) \sum_{i=1}^{n} x_i^{2k} + 4 \sum_{t=1}^{k-1} \sum_{1 \le i < j \le n}^{} \binom{2k}{2t} x_i^{2t} x_j^{2(k-t)} \\
  &= (2^{2k+1} + 4n - 2) S_{2k} + 2 \sum_{t=1}^{k-1} \binom{2k}{2t} (S_{2t} S_{2(k-t)} - S_{2k}) \\
  &= (2^{2k} + 4n + 2) S_{2k} + 2 \sum_{t=1}^{k-1} \binom{2k}{2t} S_{2t} S_{2(k-t)}
\end{align*}
where we denote $S_t = \sum_{i=1}^{n} x_i^t$. On the other hand, this is the sum of the $2k$-th powers of all elements of $\mathbb{Z}/p\mathbb{Z}$. Thus
\begin{equation} \label{eqn:main}
  (4^k + 4n + 2) S_{2k} + 2 \sum_{t=1}^{k-1} \binom{2k}{2t} S_{2t} S_{2(k-t)} = 
  \begin{cases}
    0 & \text{if } p-1 \nmid 2k, \\
    -1 & \text{if } p-1 \mid 2k.
  \end{cases}
\end{equation}

Let $a$ and $b$ be the least positive integers satisfying $p \mid 4^a + 4n + 2$ and $p \mid 4^b - 1$. Consider the set
\[
  X = \{ ax + by : x \ge 1, y \ge 0 \}.
\]
Note that the set $X$ is closed under addition. We now claim the following. 

\begin{lemma} \label{lem:1}
  If $1 \le k < (p-1)/2$ is not in $X$, then $S_{2k} = 0$. 
\end{lemma}

\begin{proof}
  We prove by induction on $k$. Suppose $S_{2k} = 0$ for all $k \le k_0 - 1$ that is not in $X$. We now show that $S_{2k_0} = 0$ if $k_0 \notin X$. Assume $k_0$ is not in $X$. Since any $k$ for which $p \mid 4^k + 4n + 2$ is of the form $a + by$ and thus in $X$, we see that $p \nmid 4^{k_0} + 4n + 2$. 

  Moreover, because $k_0 \notin X$ and $X$ is closed under addition, for each $t$ either $t$ or $k_0 - t$ is not in $X$. From Equation~\ref{eqn:main} and the induction hypothesis it follows that 
  \[
    0 = (4^{k_0} + 4n + 2) S_{2k_0} + 2 \sum_{t=1}^{k_0-1} \binom{2k_0}{2t} S_{2t} S_{2(k_0-t)} = (4^{k_0} + 4n + 2) S_{2k_0} 
  \]
  in $\mathbb{Z}/p\mathbb{Z}$. Because $4^k + 4n + 2 \neq 0$, we immediately obtain $S_{2k_0} = 0$. 
\end{proof}

Let 
\[
  e_k = \sum_{1 \le i_1 < \dots < i_k \le n}^{} x_{i_1}^2 x_{i_2}^2 \cdots x_{i_k}^2
\]
be the elementary symmetric polynomials with respect to $x_1^2, x_2^2, \dots, x_n^2$. Using a similar argument, we prove the following lemma. 

\begin{lemma}
  If $1 \le k \le n$ is not in $X$, then $e_k = 0$. 
\end{lemma}

\begin{proof}
  We again prove by induction on $k$. Suppose $e_k = 0$ for all $k \le k_0 - 1$ not in $X$, and also assume $k_0 \notin X$. The Newton identities on $x_1^2, \dots, x_n^2$ can be written as 
  \[
    k_0 e_{k_0} = e_{k_0-1} S_2 - e_{k_0-2} S_4 + \dots + (-1)^{k_0-2} e_1 S_{2(k_0-1)} + (-1)^{k_0-1} S_{2k_0}.
  \]
  Because $X$ is closed under addition and $k_0 \notin X$, for each $0 < t < k_0$ either $t \notin X$ or $k_0 - t \notin X$. From Lemma~\ref{lem:1} and the inductive hypothesis, it follows that either $e_t = 0$ or $S_{2(k_0-t)} = 0$. Therefore 
  \begin{align*}
    k_0 e_{k_0} &= e_{k_0-1} S_2 - e_{k_0-2} S_4 + \dots + (-1)^{k_0-2} e_1 S_{2(k_0-1)} + (-1)^{k_0-1} S_{2k_0} \\
    &= (-1)^{k_0-1} S_{2k_0} = 0
  \end{align*}
  and thus $e_{k_0} = 0$ since $k_0 \neq 0$.
\end{proof}

We now note that $e_n = x_1^2 \cdots x_n^2$. Since none of $x_1, \dots, x_n$ is $0$, the square of their product $e_n$ is also not $0$, and hence $n \in X$. Thus by Theorem~\ref{thm:til_alg}, $PL(n,2)$-codes exist only if $n \in X$. This finishes the proof of Theorem~\ref{thm:main}.

\section{Acknowledgments}

The author would like to express gratitude to Peter Horak, who introduced the author to the problem and provided helpful comments.

\bibliographystyle{amsplain}
\bibliography{Central}

\end{document}